\documentclass[twoside,12pt]{article}

\usepackage{amsmath}
\usepackage{amssymb}

\input xy
\xyoption{matrix}\xyoption{arrow}
\def\edge{\ar@{-}}
\def\dedge{\ar@{.}}

\long\def\ignore#1{#1}

\renewcommand{\k}{\Bbbk}

\newcommand{\N}{{\mathbb N}}
\newcommand{\mn}{{\mathbb N}}

\newcommand{\mz}{{\mathbb Z}}

\newcommand{\st}{{\rm st}}

\newcommand{\ca}{{\mathcal A}}
\newcommand{\ch}{{\mathcal H}}

\newcommand{\alambda}{\ca_\lambda}

\newcommand{\mr}{{\mathbb R}}
\newcommand{\ley}{{\rm Le}}

\newcommand{\gqmn}{{\mathcal G}_q(m,n)}

\newcommand{\goesto}{\longrightarrow} 

\def\co{{\cal O}}
\def\coq{\co_q}
\def\oqmmn{\coq(M_{m,n}(\k))}
\def\oqgmn{\coq(G_{m,n}(\k))}

\def\dhom{{\rm Dhom}}

\def\sgamma{S(\gamma)}
\def\sogamma{S^o(\gamma)}

\def\sigmac{\sigma_c}

\newcommand{\wbar}{\overline} 
\def\wr{\widehat{R}}
\def\mcgamma{{\mathcal M}_\gamma}

\def\goesto{\longrightarrow} 

\def\sse{\subseteq}
\def\ssneq{\subsetneqq} 
\def\spec{{\rm Spec}}
\def\hspec{\ch\mbox{-}\spec}

\newcommand{\fract}{\mathrm{Frac}}

\newcommand{\tq}{\,|\,}
\newcommand{\qed}{\hfill \rule{1.5mm}{1.5mm}}

\newtheorem{theorem}{Theorem}[section]
\newtheorem{proposition}[theorem]{Proposition}

\newtheorem{definition}[theorem]{Definition}
\newtheorem{lemma}[theorem]{Lemma}
\newtheorem{corollary}[theorem]{Corollary}
\newtheorem{remark}[theorem]{Remark}
\newtheorem{example}[theorem]{Example}

\newcommand{\titre}{Prime ideals in the quantum grassmannian}

\begin{document}

\title{{\vspace{-1.5cm} \bf \titre}}
\author{S Launois, T H Lenagan and L Rigal\footnote{
This research was supported by a Marie Curie Intra-European
  Fellowship within the $6^{\mbox{th}}$ European Community Framework
  Programme and by Leverhulme Research Interchange
Grant F/00158/X}}
\date{}
\maketitle

\begin{abstract} 
We consider quantum Schubert cells in the quantum grassmannian and give a cell
decomposition of the prime spectrum via the Schubert cells. As a consequence,
we show that all primes are completely prime in the generic case where the
deformation parameter $q$ is not a root of unity. There is a natural torus
action of $\ch = (\k^*)^n$ on $\gqmn$ and the cell decomposition of the set of
$\ch$-primes leads to a parameterisation of the $\ch$-spectrum via certain
diagrams on partitions associated to the Schubert cells. Interestingly, the
same parameterisation occurs for the non-negative cells in recent studies
concerning the totally non-negative grassmannian. Finally, we use the cell
decomposition to establish that the quantum grassmannian satisfies normal
separation and catenarity.
\end{abstract}




\section*{Introduction}

Let $m \le n$ be positive integers and let $\oqmmn$ denote the quantum
deformation of the affine coordinate ring on $m \times n$ matrices, with
nonzero deformation parameter $q$ in the base field. The quantum deformation
of the homogeneous coordinate ring of the grassmannian, denoted $\oqgmn$, is
defined as the subalgebra of $\oqmmn$ generated by the maximal quantum minors
of the generic matrix of $\oqmmn$. To simplify, these algebras will be
referred to in the sequel as the algebra of quantum matrices and the quantum
grassmannian, respectively. 

The main goal of this work is the study of the prime spectrum of the quantum
grassmannian. This algebra is naturally endowed with the action of a torus
$\ch$. Thus, according to the philosophy of the {\em stratification theory} as
developed by Goodearl and Letzter (see \cite{bg}), our main concern is the set
of $\ch$-prime ideals (namely, the prime ideals invariant under the action of
$\ch$). Recall that if $A$ is an algebra and $\ch$ a torus which acts on $A$
by algebra automorphisms then the stratification theory suggests a study of
the prime spectrum of $A$ by means of a partition into strata, each stratum
being indexed by an $\ch$-prime ideal. For many algebras arising from the
theory of quantum groups, general results have been proved about such a
stratification. For example, when such an algebra is a certain kind of
iterated skew polynomial extension, general results show that it has only
finitely many $\ch$-primes and that each stratum is homeomorphic to the
spectrum of a suitable commutative Laurent polynomial ring. However, the
algebra which interests us here is far from being such an extension and it is
not even clear at the outset that it has finitely many $\ch$-primes. For this
reason, these general results do not apply and we are led to use a very
different approach which has a geometric flavour. Recall that a classical
approach to the study of the grassmanian variety $G_{m,n}(\k)$ is to use its
partition into Schubert cells and their closures which are the so-called
Schubert subvarieties of the grassmannian. Notice that, in this decomposition,
Schubert cells are indexed by Young diagrams fitting in a rectangular $m
\times (n-m)$ Young diagram. Our method is inspired by this classical
geometric setting. 

Quantum analogues of Schubert varieties (or rather of their coordinate rings)
were studied in \cite{lr2} in order to show that the quantum grassmannian has
a certain combinatorial structure, namely the stucture of a {\em quantum
graded algebra with a straightening law}. Subsequently, some of their
properties have been established in \cite{lr3}. In this paper, we define
quantum Schubert cells as noncommutative dehomogenisations of quantum Schubert
varieties. Using the structure of a quantum graded algebra with a
straightening law enjoyed by the quantum grassmannian, we are then in position
to define a partition of its prime spectrum. This partition is called a {\em
cell decomposition} since it turns out that the set of $\ch$-primes of a given
component is in natural one-to-one correspondence with the set of $\ch$-primes
of an associated quantum Schubert cell. Hence, the description of the
$\ch$-primes of the quantum grassmannian reduces to that of the $\ch$-primes
of each of its associated quantum Schubert cells. (Here, the actions of $\ch$
on the quantum Schubert varieties and cells are naturally induced by its
action on the quantum grassmannian.)

On the other hand, we can show that a quantum Schubert cell can be identified
as a subalgebra of a quantum matrix algebra, with the variables that are
included sitting naturally in the Young diagram associated to that cell. As a
consequence, we can establish properties for quantum Schubert cells akin to
known properties of quantum matrix algebras. For example, we are able to
parameterise the $\ch$-prime ideals of a quantum Schubert cell by {\em Cauchon
diagrams} on the corresponding Young diagram, in the same way that Cauchon was
able to parameterise the $\ch$-prime ideals in quantum matrices, see
\cite{c2}. This is achieved by using the theory of {\em deleting derivations}
as developed by Cauchon in \cite{c1}. This theory utilizes certain changes of
variables in the field of fractions of the algebra under consideration. In the
case of quantum matrices, these changes of variable can be reinterpreted using
quasi-determinants, see \cite{gelret}. Recently, Cauchon diagrams in Young
diagrams have appeared in the literature under the name Le-diagrams see, for
example, \cite{post} and \cite{w}.

By using this approach, we are able to show that there are only finitely many
$\ch$-prime ideals in $\oqgmn$. More precisely, we show that such $\ch$-primes
are in natural one-to-one correspondence with Cauchon diagrams defined on
Young diagrams fitting into a rectangular $m \times (n-m)$ Young diagram.
Following on from this description, we are able to calculate the number of
$\ch$-prime ideals in the quantum grassmannian. 

In addition, we are able to show that prime ideals in the quantum grassmannian
are completely prime, and that this algebra satisfies normal separation and,
hence, is catenary. Again, the method is to establish these properties for
each quantum Schubert cell and then transfer them to the quantum grassmannian.

To conclude this introduction, it should be stressed that there are very
interesting connections between our results in the present paper and recent
results in the theory of total positivity. More details on this are given in
Section 5.

\section{Basic definitions}\label{sec-basic-definitions}

Throughout the paper, $\k$ is a field and $q$ is a nonzero element of $\k$
that is not a root of unity. Occasionally, we will remind the reader of this
restriction in the statement of results. 

In this section, we collect some basic definitions and properties about the
objects we intend to study. Most proofs will be omitted since these results
already appear in \cite{klr, lr2, lr3}. Appropriate references will be given
in the text.\\

Let $m,n$ be positive integers. \\

The quantisation of the coordinate ring of the affine variety $M_{m,n}(\k)$ of
$m \times n$ matrices with entries in $\k$ is denoted ${\mathcal
O}_q(M_{m,n}(\k))$. It is the $\k$-algebra generated by $mn$ indeterminates
$x_{ij}$, with $1 \le i \le m$ and $1 \le j \le n$, subject to the relations:
\[
\begin{array}{ll}  
x_{ij}x_{il}=qx_{il}x_{ij},&\mbox{ for }1\le i \le m,\mbox{ and }1\le j<l\le
n\: ;\\ x_{ij}x_{kj}=qx_{kj}x_{ij}, & \mbox{ for }1\le i<k \le m, \mbox{ and }
1\le j \le n \: ; \\ x_{ij}x_{kl}=x_{kl}x_{ij}, & \mbox{ for } 1\le k<i \le m,
\mbox{ and } 1\le j<l \le n \: ; \\
x_{ij}x_{kl}-x_{kl}x_{ij}=(q-q^{-1})x_{il}x_{kj}, & \mbox{ for } 1\le i<k \le
m, \mbox{ and } 1\le j<l \le n.
\end{array}
\]
To simplify, we write $M_{n}(\k)$ for $M_{n,n}(\k)$ and 
${\mathcal O}_q(M_{n}(\k))$ for ${\mathcal O}_q(M_{n,n}(\k))$.
The $m \times n$ matrix
${\bf X}=(x_{ij})$ is called the generic matrix associated with
${\mathcal O}_q(M_{m,n}(\k))$.
\\

As is well known, there exists a $\k$-algebra {\em transpose isomorphism}
between ${\mathcal O}_q(M_{m,n}(\k))$ and ${\mathcal O}_q(M_{n,m}(\k))$, see
\cite[Remark 3.1.3]{lr2}. Hence, from now on, we assume that $m \le n$,
without loss of generality.\\

An index pair is a pair $(I,J)$ such
that $I \subseteq \{1,\dots,m\}$ and $J \subseteq \{1,\dots,n\}$ are subsets
with the same cardinality. Hence, an index pair is given by an integer $t$
such that $1 \le t \le m$ and ordered sets 
$I=\{i_1 < \dots < i_t\} \subseteq \{1,\dots,m\}$
and $J=\{j_1 < \dots < j_t\} \subseteq \{1,\dots,n\}$. To any such index pair
we associate the quantum minor 
\[ 
[I|J] = \sum_{\sigma\in {\mathfrak S}_t}
(-q)^{\ell(\sigma)} x_{i_{\sigma(1)}j_1} \dots x_{i_{\sigma(t)}j_t} . 
\]

\begin{definition} -- \label{def-q-grassmannian} 
The {\it quantisation of the coordinate ring of the grassmannian of
$m$-dimensional subspaces of $\k^n$}, denoted by ${\mathcal O}_q(G_{m,n}(\k))$
and informally referred to as the ($m\times n$) quantum grassmannian is the
subalgebra of ${\mathcal O}_q(M_{m,n}(\k))$ generated by the $m \times m$
quantum minors.
\end{definition} 

An index set  is a subset $I=\{i_1 < \dots < i_m\}
\subseteq \{1,\dots,n\}$. To any index set we associate the maximal quantum
minor $[I]:=[\{1,\dots,m\}|I]$ of ${\mathcal O}_q(M_{m,n}(\k))$ 
which is, thus, an
element of ${\mathcal O}_q(G_{m,n}(\k))$. The set of all index sets is denoted
by $\Pi_{m,n}$. Since $\Pi_{m,n}$ is in one-to-one correspondence with the set
of all maximal quantum minors of ${\mathcal O}_q(M_{m,n}(\k))$, we will often
identify these two sets. We equip  $\Pi_{m,n}$ with a partial order 
$\le_\st$ defined in the following way. Let $I=\{i_1< \dots
<i_m\}$ and $J=\{j_1< \dots <j_m\}$ be  two index sets, then 
\[ 
I\le_\st J
\Longleftrightarrow i_s \le j_s \quad\mbox{for}\quad 1 \le s \le m. 
\]


For example, Figure~\ref{gra} shows the partial ordering on generators of
${\mathcal O}_q(G_{3,6}(\k))$.\\ 


Let $A$ be a noetherian $\k$-algebra, and assume that the torus
$\ch:=(\k^*)^r$ acts rationally on $A$ by $\k$-algebra 
automorphisms. (For details
concerning rational actions of tori, see \cite[Chapter II.2]{bg}.) A two-sided
ideal $I$ of $A$ is said $\ch$-invariant if $h\cdot I=I$ for all $h \in \ch$.
An $\ch$-prime ideal of $A$ is a proper $\ch$-invariant ideal $J$ of $A$ such
that whenever $J$ contains the product of two $\ch$-invariant ideals of $A$
then $J$ contains at least one of them. We denote by $\ch$-$\spec(A)$ the
$\ch$-spectrum of $A$; that is, the set of all $\ch$-prime ideals of $A$. It
follows from \cite[Proposition II.2.9]{bg} that every $\ch$-prime ideal is
prime when $q$ is not a root of unity; so that in this case $\ch$-$\spec(A)$
coincides with the set of all $\ch$-invariant prime ideals of $A$.

There are natural torus actions on the classes of algebras that we study here,
including quantum matrices, partition subalgebras of quantum matrices and
quantum grassmannians. These actions are rational; and so the remarks above
apply.


First, there is an action of a torus $\ch := (\k^*)^{m+n}$ on $\oqmmn$ given
by 
\[
(\alpha_1, \dots, \alpha_m, \beta_1, \dots, \beta_n)\circ x_{ij} :=
\alpha_i\beta_jx_{ij}.
\] 
In other words, one is able to multiply through rows
and columns by nonzero scalars.


Next, there is an action of the torus $\ch:= (\k^*)^n$ on ${\mathcal
O}_q(G_{m,n}(\k))$ which comes from the column action on quantum
matrices. Thus, $(\alpha_1,\dots,\alpha_n)\circ[i_1,\dots,i_m] :=
\alpha_{i_1}\dots\alpha_{i_m}[i_1,\dots,i_m]$. We shall be interested in prime
ideals left invariant under the action of this torus. The set of such prime
ideals is the $\ch$-spectrum of ${\mathcal O}_q(G_{m,n}(\k))$.
\\

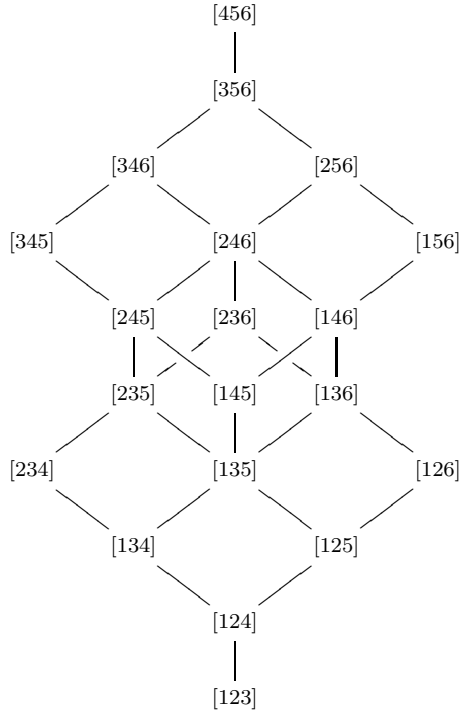
\begin{figure}[ht]
\ignore{
$$\xymatrixrowsep{2.4pc}\xymatrixcolsep{3.2pc}\def\objectstyle{\scriptstyle}
\xymatrix@!0{
 && [ 456 ] \edge[d]\\
 && [ 356] \edge[dl] \edge[dr]\\
 & [ 346] \edge[dl] \edge[dr] &&[ 256] \edge[dl] \edge[dr]\\
[ 345] \edge[dr] && [ 246] \edge[dl] \edge[dr] \edge[d] && [ 156] 
  \edge[dl]\\
 & [ 245] \edge[d] \edge[dr] & [ 236] \edge[dl]|\hole \edge[dr]|\hole &[146]
\edge[dl] \edge[d]\\
 &[235] \edge[dl] \edge[dr] &[145] \edge[d] &[136] \edge[dl] \edge[dr]\\
[234] \edge[dr] &&[135] \edge[dl] \edge[dr] &&[126]\edge[dl]\\
 &[134] \edge[dr] &&[125] \edge[dl]\\
 &&[124] \edge[d]\\
 &&[123]
}$$
}
\caption{The partial ordering $\le_\st$ on 
${\mathcal O}_q(G_{3,6}(\k))$.}\label{gra}
\end{figure}

We recall the definition of quantum Schubert varieties given in 
\cite{lr3}. 

\begin{definition} -- \label{def-q-schubert}
Let $\gamma\in\Pi_{m,n}$ and put $\Pi_{m,n}^\gamma
=\{\alpha\in\Pi_{m,n} \tq \alpha\not\ge_\st\gamma\}$. 
The quantum Schubert variety $\sgamma$ associated to
$\gamma$ is 
\[ 
\sgamma:= {\mathcal
O}_q(G_{m,n}(\k))/\langle \Pi_{m,n}^\gamma \rangle . 
\] 
(Note that $\sgamma$ was denoted by ${\mathcal O}_q(G_{m,n}(\k))_\gamma$ in
\cite{lr3}.)

\end{definition}

This definition is inspired by the classical description of the coordinate
rings of Schubert varieties in the grassmannian. For more details about this
matter, see \cite[Section 6.3.4]{glak}.\\

Note that each of the maximal quantum minors that generate $\oqgmn$ is an
$\ch$-eigenvector. Thus, the $\ch$-action on $\oqgmn$ transfers to 
the quantum Schubert varieties $\sgamma$. \\

In order to study properties of the quantum grassmannian, the notion of a
quantum graded algebra with a straightening law (on a partially ordered set
$\Pi$) was introduced in \cite{lr2}. We now recall the definition of these 
algebras and mention various properties that we will
use later.\\

Let $A$ be an algebra and $\Pi$ a finite subset of 
elements of $A$ with a partial order $<_\st$. A  
{\em standard monomial} on $\Pi$ is an element
of $A$ which is either $1$ or of the form $\alpha_1\dots\alpha_s$, 
for some $s\geq 1$, where $\alpha_1,\dots,\alpha_s \in \Pi$ and
$\alpha_1\le_\st\dots\le_\st\alpha_s$. 

\begin{definition} \label{recall-q-gr-asl} -- 
Let $A$ be an ${\mathbb N}$-graded $\k$-algebra and $\Pi$ a finite subset 
of $A$ equipped with a partial order $<_\st$. 
We say that $A$ is a {\em quantum graded algebra with a straightening law} 
({\em quantum graded A.S.L.} for short) on the poset $(\Pi,<_\st)$ 
if the following conditions are satisfied.\\
(1) The elements of $\Pi$ are homogeneous with positive degree.\\
(2) The elements of $\Pi$ generate $A$ as a $\k$-algebra.\\
(3) The set of standard monomials on $\Pi$ is a linearly independent set.\\
(4) If $\alpha,\beta\in\Pi$ are not comparable for $<_\st$, 
then $\alpha\beta$ 
is a linear combination of terms $\lambda$ or $\lambda\mu$, where 
$\lambda,\mu\in\Pi$, $\lambda\le_\st\mu$ and $\lambda<_\st\alpha,\beta$.\\
(5) For all $\alpha,\beta\in\Pi$, there exists $c_{\alpha\beta} \in \k^\ast$ 
such that $\alpha\beta-c_{\alpha\beta}\beta\alpha$ is a linear combination of 
terms $\lambda$ or $\lambda\mu$, where $\lambda,\mu\in\Pi$,
$\lambda\le_\st\mu$ and $\lambda<_\st\alpha,\beta$.
\end{definition}

By \cite[Proposition 1.1.4]{lr2}, if $A$ is a quantum graded A.S.L. on the
partially 
ordered set $(\Pi,<_\st)$, then the set of standard monomials on $\Pi$ forms a
$\k$-basis of $A$. Hence, in the presence of a standard monomial basis, the
structure of a quantum graded A.S.L. may be seen as providing more detailed 
information on the way standard monomials multiply and commute.

\begin{example} -- \rm
It is shown, in \cite[Theorem 3.4.4]{lr2}, that ${\mathcal O}_q(G_{m,n}(\k))$
is a quantum graded algebra with a straightening law on $(\Pi_{m,n},\le_\st)$.
\end{example}

From our point of view, one important feature of quantum graded A.S.L. 
is the
following. Let $A$ be a $\k$-algebra which is a quantum graded A.S.L. on the
set $(\Pi,\le_\st)$. A subset $\Omega$ of $\Pi$ will be called a $\Pi$-ideal
if it is an ideal of the partially 
ordered set $(\Pi,\le_\st)$ in the sense of lattice
theory; that is, if it satisfies the following property: if $\alpha\in\Omega$
and if $\beta\in\Pi$, with $\beta \le_\st \alpha$, then $\beta\in\Omega$. We
can consider the quotient $A/\langle\Omega\rangle$ of $A$ by the ideal
generated by $\Omega$. Clearly, it is still a graded algebra and it is
generated by the images in $A/\langle\Omega\rangle$ of the elements of
$\Pi\setminus\Omega$. The important point here is that
$A/\langle\Omega\rangle$ inherits from $A$ 
a natural quantum graded A.S.L. structure on
$\Pi\setminus\Omega$ (or, more precisely, 
on the canonical image of $\Pi\setminus\Omega$
in $A/\langle\Omega\rangle$). In particular, the set of
homomorphic images in $A/\langle\Omega\rangle$ of the standard monomials of
$A$ which either equal $1$ or are of the form $\alpha_1 \dots \alpha_t$
($t\in\N^\ast$) and $\alpha_1\notin\Omega$ form a $\k$-basis for
$A/\langle\Omega\rangle$. The reader will find all the necessary details 
in \S
1.2 of \cite{lr2}.

\begin{example} -- \rm \label{example-q-gr-ASL}
Let $\gamma \in\Pi_{m,n}$. It is clear that the set $\Pi_{m,n}^\gamma$
introduced in Definition \ref{def-q-schubert} is a $\Pi_{m,n}$-ideal. Hence,
the discussion above shows that the quantum Schubert variety $S(\gamma)$ is a
quantum graded A.S.L. on the canonical image in $S(\gamma)$ of $\Pi_{m,n}
\setminus \Pi_{m,n}^\gamma$. In particular, the canonical images in
$S(\gamma)$ of the standard monomials of ${\mathcal O}_q(G_{m,n}(\k))$ which
either equal to $1$ or are of the form $[I_1]\dots[I_t]$, for some $t\geq 1$
and with $\gamma \le_\st [I_1]$, form a $\k$-basis of $S(\gamma)$.
\end{example}

\begin{remark} -- \rm \label{remark-single-min-elt}
Let $\gamma \in\Pi_{m,n}$. As mentioned in Example \ref{example-q-gr-ASL}, the
quantum Schubert variety $S(\gamma)$ is a quantum graded A.S.L. on the
canonical image in $S(\gamma)$ of $\Pi_{m,n} \setminus \Pi_{m,n}^\gamma$. At
this point, it is worth noting that the set $\Pi_{m,n} \setminus
\Pi_{m,n}^\gamma$ has a single minimal element, namely $\gamma$, and that the
image of $\gamma$ is a normal nonzerodivisor in $S(\gamma)$, by \cite[Lemma
1.2.1]{lr2}.
\end{remark}


\section{Partition subalgebras of quantum matrices}
\label{section-partitions}


Let $\lambda = (\lambda_1, \lambda_2, \dots, \lambda_m)$ be a partition with
$n\geq \lambda_1 \geq \lambda_2\geq \dots \geq \lambda_m \geq 0$. The
{\em partition subalgebra $\alambda$} 
of $\oqmmn$ is defined to be the subalgebra
of $\oqmmn$ generated by the variables $x_{ij}$ with $j\leq \lambda_i$. By
looking at the defining relations for quantum matrices, it is easy to see that
$\alambda$ can be presented as an iterated Ore extension with the variables
$x_{ij}$ added in lexicographic order. As a consequence, partition
subalgebras are noetherian domains.  
Recall that there is an action of a torus $\ch := (\k^*)^{m+n}$ on 
$\oqmmn$ given by
$(\alpha_1, \dots, \alpha_m, \beta_1, \dots, \beta_n)\circ x_{ij} :=
\alpha_i\beta_jx_{ij}.$ This action induces an action on $\alambda$, by
restriction. 
Our
main aim in this section is to observe that the Goodearl-Letzter
stratification theory and the Cauchon theory of deleting derivations apply to
partition subalgebras of quantum matrices. As a consequence, we can then
exploit these theories to obtain information about the prime and $\ch$-prime
spectra of partition subalgebras.

The conditions needed to use the theories have been brought together in the
notion of a (torsion-free) CGL-extension introduced in 
\cite[Definition 3.1]{llr}; the definition is given below, for convenience.

\begin{definition}\label{def-cgl}
{\rm  An iterated skew polynomial extension 
\[
A = \k[x_1][x_2; \sigma_2, \delta_2] \dots [x_n; \sigma_n,\delta_n]
\]
is said to be a {\em CGL extension} (after Cauchon, Goodearl and Letzter)  
provided that the following list of
conditions is satisfied:

\begin{itemize}

\item With $A_j:= \k[x_1][x_2; \sigma_2, \delta_2] \dots [x_j;
\sigma_j,\delta_j]$ for each $1\leq j\leq n$, each $\sigma_j$ is a
$\k$-algebra automorphism of $A_{j-1}$, 
each $\delta_{j}$ is a locally nilpotent
$\k$-linear $\sigma_{j}$-derivation of
$A_{j-1}$, and there exist nonroots of unity $q_j \in \k^*$ with
$\sigma_j\delta_j = q_j\delta_j\sigma_j$;

\item For each $i<j$ there exists a $\lambda_{ji}\in\k^*$ such that
$\sigma_j(x_i) = \lambda_{ji}x_i$;

\item There is a torus $\ch = (\k^*)^r$ acting rationally on $A$ by 
$\k$-algebra
automorphisms;

\item The $x_i$ for $1\leq i\leq n$ are $\ch$-eigenvectors;

\item There exist elements $h_1, \dots, h_n \in \ch$ such that $h_j(x_i) =
\sigma_j(x_i)$ for $j>i$ and such that the $h_j$-eigenvalue of $x_j$ is not a
root of unity. 

\end{itemize}

If, in addition, the subgroup of $\k^*$ generated by the $\lambda_{ji}$ is
torsionfree then we will say that $A$ is a {\em torsionfree CGL extension}.}
\end{definition}

For a discussion of rational actions of tori, see \cite[Chapter II.2]{bg}.
\\

It
is easy to check that all of these conditions are satisfied for partition
subalgebras (for exactly the same reasons that quantum matrices are
CGL-extensions).

\begin{proposition} \label{prop-partition-cgl}
Partition subalgebras of quantum matrix algebras 
are CGL-extensions and are torsion-free CGL
extensions when the parameter $q$ is not a root of unity.
\end{proposition}

\begin{proof}
It is only necessary to show that we can introduce the variables $x_{ij}$ that
define the partition subalgebra in such a way that the resulting
iterated skew polynomial extension satisfies the list of conditions above.
Lexicographic ordering is suitable. 
\qed\end{proof}

\begin{corollary}
Let $\alambda$ be a partition subalgebra of quantum matrices and suppose that
$\alambda$ is equipped with the induced action of $\ch$. Then $\alambda$ has
only finitely many $\ch$-prime ideals and all prime ideals of $\alambda$ are
completely prime when the parameter $q$ is not a root of unity.
\end{corollary}

\begin{proof}
This follows immediately from the previous result and 
\cite[Theorem II.5.12 and Theorem II.6.9]{bg}.
\qed\end{proof}\\

In fact, we can be much more precise about the number of $\ch$-primes. We will
prove below that there exists a natural bijection between the $\ch$-prime
spectrum of $\alambda$ and Cauchon diagrams defined on the Young diagram
corresponding to the partition $\lambda$.

Suppose that $Y_\lambda$ is the Young diagram corresponding to the partition
$\lambda$. Then a {\em Cauchon diagram} on $Y_\lambda$ is an assignment of a
colour, either white or black, to each square of the diagram $Y_\lambda$ in
such a way that if a square is coloured black then either each square above is
coloured black, or each square to the left is coloured black. These diagrams
were first introduced by Cauchon, \cite{c2}, in his study of the $\ch$-prime
spectrum of quantum matrices. Recently, they have occurred with the name
$\ley$-diagrams in work of Postnikov, \cite{post}, and Williams, \cite{w}. 
\\

\begin{lemma} Let $\lambda = (\lambda_1, \lambda_2, \dots, \lambda_m)$ 
be a partition with
$n\geq \lambda_1 \geq \lambda_2\geq \dots \geq \lambda_m > 0$. The number of
$\ch$-prime ideals in $\alambda$ is equal to the number of Cauchon diagrams
defined on the Young diagram corresponding to the partition $\lambda$.
\end{lemma}

\begin{proof} Let $n_{\lambda}$ denote the number of $\ch$-prime ideals in
$\alambda$. First, we obtain a  recurrence relation for $n_{\lambda}$. 

The $\ch$-prime
spectrum of $\alambda$ can be written as a disjoint union:
$$\hspec(\alambda)=\{ J \in \hspec(\alambda) | x_{m, \lambda_m} \in J \}
\sqcup \{ J \in \hspec(\alambda) | x_{m, \lambda_m} \notin J \}.$$

It follows from the complete primeness of every $\ch$-prime ideal of
$\alambda$ that an $\ch$-prime ideal $J$ of $\alambda$ that contains $x_{m,
\lambda_m}$ must also contain either $x_{i, \lambda_m}$ for each $i \in \{1,
\dots, m\}$ or $x_{m, \alpha}$ for each $\alpha \in \{1, \dots , \lambda_m\}$.
Let $I_1$ be the ideal generated by $x_{i, \lambda_m}$ for $i \in \{1, \dots,
m\}$, and let $I_2$ be the ideal generated by $x_{m, \alpha}$ for $\alpha \in
\{1, \dots , \lambda_m\}$. Set $I_3:=I_1 +I_2$. As 
$$\frac{\alambda}{I_1} \simeq 
\ca_{(\lambda_1-1, \lambda_2-1, \dots , \lambda_m-1)}, \ \
\frac{\alambda}{I_2} \simeq 
\ca_{(\lambda_1, \lambda_2, \dots , \lambda_{m-1})} 
\mbox{~~and~~} \frac{\alambda}{I_3} 
\simeq \ca_{(\lambda_1-1, \lambda_2-1, \dots , \lambda_{m-1}-1)}, 
$$
we obtain 

\begin{eqnarray*}
n_{\lambda}&=& n_{(\lambda_1-1, \lambda_2-1, \dots , \lambda_m-1)} +
n_{(\lambda_1, \lambda_2, \dots , \lambda_{m-1})} -n_{(\lambda_1-1,
\lambda_2-1, \dots , \lambda_{m-1}-1)} \\&&
+ \; |\{ J \in \hspec(\alambda) | x_{m, \lambda_m} \notin J \}|.
\end{eqnarray*}
(Even though the above isomorphisms are not always $\ch$-equivariant, they
preserve the property of being an $\ch$-prime.) 

As $\alambda$ is a CGL extension, one can apply the theory of deleting
derivations to this algebra. In particular, it follows from \cite[Th\'eor\`eme
3.2.1]{c1} that the multiplicative system of $\alambda$ generated by
$x_{m,\lambda_m}$ is an Ore set in $\alambda$, and
$$
\alambda[x_{m,\lambda_m}^{-1}] \simeq \ca_{(\lambda_1, \lambda_2, \dots ,
\lambda_{m-1}, \lambda_m -1)} [y^{\pm 1};\sigma],
$$ 
where $\sigma$ is the
automorphism of $\ca_{(\lambda_1, \lambda_2, \dots , \lambda_{m-1}, \lambda_m
-1)}$ defined by $\sigma(x_{i\alpha})=q^{-1}x_{i\alpha}$ if $i=m$ or
$\alpha=\lambda_m$, and $\sigma(x_{i\alpha})=x_{i\alpha}$ otherwise. 
Denote this isomorphism by $\psi$, and note that 
$\psi(x_{m,\lambda_m}) = y$. As $x_{m,\lambda_m}$
is an $\ch$-eigenvector, the action of $\ch$ on $\alambda$ extends to an
action of $\ch$ on $\alambda[x_{m,\lambda_m}^{-1}] $, and so on
$\ca_{(\lambda_1, \lambda_2, \dots , \lambda_{m-1}, \lambda_m -1)} [y^{\pm 1};
\sigma]$. 
It is easy to show 
that this action restricts to an action on $\ca_{(\lambda_1, \lambda_2,
\dots , \lambda_{m-1}, \lambda_m -1)}$ which coincides with the ``natural''
action of $\ch$ on this algebra.
Hence the isomorphism
$\psi$ induces a bijection from $\{ J \in \hspec(\alambda) | x_{m, \lambda_m}
\notin J \}$ to $\hspec (\ca_{(\lambda_1, \lambda_2, \dots , \lambda_{m-1},
\lambda_m -1)} [y^{\pm 1}; \sigma])$; and so it follows from \cite[Theorem
2.3]{llr} that there exists a bijection between $\{ J \in \hspec(\alambda) |
x_{m, \lambda_m} \notin J \}$ and $\hspec (\ca_{(\lambda_1, \lambda_2, \dots ,
\lambda_{m-1}, \lambda_m -1)}) $. Hence 
$$
|\{ J \in \hspec(\alambda) | x_{m,
\lambda_m} \notin J \}|= n_{(\lambda_1, \lambda_2, \dots , \lambda_{m-1},
\lambda_m -1)};
$$ 
so that 
$$
n_{\lambda}= n_{(\lambda_1-1, \lambda_2-1, \dots ,
\lambda_m-1)} + n_{(\lambda_1, \lambda_2, \dots , \lambda_{m-1})}
-n_{(\lambda_1-1, \lambda_2-1, \dots , \lambda_{m-1}-1)} + n_{(\lambda_1,
\lambda_2, \dots , \lambda_{m-1}, \lambda_m -1)}. 
$$

On the other hand, it follows from \cite[Remark 4.2]{w} that the number of
Cauchon diagrams (equivalently, Le-diagrams) defined on the Young diagram
corresponding to the partition $\lambda$ satisfies the same recurrence. As the
number of $\ch$-prime ideals in $\ca_{(1)}$ is equal to $2$ which is also the
number of Cauchon diagrams defined on the Young diagram corresponding to the
partition $\lambda = (1)$, the proof is complete. \qed\end{proof}\\

Let $\lambda = (\lambda_1, \lambda_2, \dots, \lambda_m)$ be a partition with
$n\geq \lambda_1 \geq \lambda_2\geq \dots \geq \lambda_m > 0$ and let
$\alambda$ be the corresponding partition subalgebra of generic quantum
matrices. Let ${\mathcal C}_\lambda$ denote the set of Cauchon diagrams on the
Young diagram $Y_\lambda$ corresponding to the partition $\lambda$. We have
just seen that the sets $\hspec(\alambda)$ and ${\mathcal C}_\lambda$ have the
same cardinality. In fact, there is a natural bijection between these two sets
which carries over important algebraic and geometric information. This natural
bijection arises by using Cauchon's theory of deleting derivations developed
in \cite{c1} and \cite{c2}.

As $\alambda$ is a CGL extension, the theory of deleting derivations can be
applied to the iterated Ore extension $\alambda=k[x_{1,1}]\dots
[x_{m,\lambda_m};\sigma_{m,\lambda_m},\delta_{m,\lambda_m}]$ (where the
indices are increasing for the lexicographic order). Before describing the
deleting derivations algorithm, we introduce some notation. Denote by
$\leq_{{\rm lex}}$ the lexicographic ordering on $\mathbb{N}^2$ and set
$E:=\left( \bigsqcup_{i=1}^m \{i\} \times \{1, \dots , \lambda_i \} \cup
\{(m,\lambda_m+1)\} \right) \setminus \{(1,1)\}$. If $(j,\beta) \in E$ with
$(j,\beta) \neq (m,\lambda_m+1)$, then $(j,\beta)^{+}$ denotes the least
element (relative to $\leq_{{\rm lex}}$) of the set $\left\{ (i,\alpha) \in E
\mbox{ $\mid$}(j,\beta) < (i,\alpha) \right\}$. 

The deleting derivations algorithm constructs, for each $r \in E$, a 
family of elements $x_{i,\alpha}^{(r)}$ for $\alpha \leq \lambda_i$  
of $F:=\fract(\alambda)$, defined as
follows. \\$ $
\begin{enumerate}
\item \underline{If $r=(m,\lambda_m+1)$}, then
$x_{i,\alpha}^{(m,\lambda_m+1)}=x_{i,\alpha}$ for all $(i,\alpha)$ 
with $\alpha \leq \lambda_i$.\\$ $
\item \underline{Assume that $r=(j,\beta) < (m,\lambda_m+1)$}
and that the $x_{i,\alpha}^{(r^{+})}$ are already constructed.
Then, it follows from \cite[Th\'eor\`eme 3.2.1]{c1} that
$x_{j,\beta}^{(r^+)} \neq 0$ and,
for all $(i,\alpha)$, we have:
$$x_{i,\alpha}^{(r)}=\left\{ \begin{array}{ll}
x_{i,\alpha}^{(r^{+})}-x_{i,\beta}^{(r^{+})}
\left(x_{j,\beta}^{(r^{+})}\right)^{-1}
x_{j,\alpha}^{(r^{+})}
& \mbox{ if } i<j \mbox{ and } \alpha < \beta \\
x_{i,\alpha}^{(r^{+})} & \mbox{ otherwise.}
\end{array} \right.$$
\end{enumerate}

As in \cite{c1}, we denote by $\overline{\alambda}$ the subalgebra of
$\fract(\alambda)$ generated by the indeterminates obtained at the end of this
algorithm; that is, we denote by $\overline{\alambda}$ the subalgebra of
$\fract(\alambda)$ generated by the $t_{i,\alpha}:=x_{i,\alpha}^{(1,2)}$ for
each $(i,\alpha)$ such that $\alpha \leq \lambda_i$. Cauchon has shown that
$\overline{\alambda}$ can be viewed as the quantum affine space
$\overline{\alambda}$ generated by indeterminates $t_{ij}$ for $j \leq
\lambda_i$ with relations $t_{ij}t_{il} = qt_{il}t_{ij}$ for $j<l$, while 
$t_{ij}t_{kj} = qt_{kj}t_{ij}$ for $i<k$, and all other pairs commute. 
Observe that the torus $\ch$ still acts by
automorphisms on $\overline{\alambda}$ via $(a_1,\dots,a_m,b_1, \dots, b_n).
t_{ij}=a_ib_j t_{ij}$. The theory of deleting derivations allows the
explicit (but technical) construction of an embedding $\varphi$, called the
canonical embedding, from $\hspec(\alambda)$ into the $\ch$-prime spectrum of
$\overline{\alambda}$. The $\ch$-prime ideals of $\overline{\alambda}$ are
well-known: they are generated by the subsets of
$\{t_{ij}\}$. If $C$ is a Cauchon diagram defined on the Young tableau
corresponding to $\lambda$, then we denote by $K_C$ the (completely) prime
ideal of $\overline{\alambda}$ generated by the subset of indeterminates 
$t_{ij}$ such that the square in position $(i,j)$ is a black square of $C$.

\begin{theorem}
\label{theopartitionalgebra}
Let $\lambda = (\lambda_1, \lambda_2, \dots, \lambda_m)$ be a partition with
$n\geq \lambda_1 \geq \lambda_2\geq \dots \geq \lambda_m > 0$ and let
$\alambda$ be the corresponding partition subalgebra of generic quantum
matrices. Let $\mathcal{C}_{\lambda}$ denote the set of Cauchon diagrams
defined on the Young tableau corresponding to $\lambda$. \\ 
For every Cauchon
diagram $C \in \mathcal{C}_\lambda$, there exists a unique $\ch$-invariant
(completely) prime ideal $J_C$ of $\alambda$ such that $\varphi (J_C)=K_C$. 
Moreover there is no other $\ch$-prime in $\alambda$; so that 
$$
\hspec (\alambda)=\{J_C | C \in \mathcal{C}_{\lambda} \}.
$$
\end{theorem}

\begin{proof} As the sets $\hspec (\alambda)$ and $\{ J_C | C \in
\mathcal{C}_{\lambda} \}$ have the same cardinality by the previous lemma, it
is sufficient to show that $\hspec (\alambda) \subseteq \{ J_C | C \in
\mathcal{C}_{\lambda} \}$. This inclusion can be obtained  by following the
arguments of \cite[Lemmes 3.1.6 and 3.1.7]{c2}. The details are left to the 
interested reader. \qed\end{proof}

\begin{remark}
{\rm Theorem \ref{theopartitionalgebra} provides more than just an explicit
bijection between the $\ch$-spectrum of $\alambda$ and
$\mathcal{C}_{\lambda}$. This natural bijection carries algebraic and
geometric data. For example, it can be shown that the height of $J_C$ is given
by the number of black boxes of the Cauchon diagram $C$. Also, the dimension
of the $\ch$-stratum (in the sense of \cite[Definition 2.2.1]{bg}) associated
to $J_C$ can be read off from the Cauchon diagram $C$.}
\end{remark}

An algebra $A$ is said to be {\em catenary} if for each pair of prime ideals
$Q\sse P$ of $A$ all saturated chains of prime ideals between $Q$ and $P$ have
the same length. Our next aim is to show that partition subalgebras of quantum
matrix algebras are catenary. The key property that we need to establish in
order to prove catenarity is the property of normal separation. Two prime
ideals $Q\ssneq P$ are said to be {\em normally separated} if there is an
element $c\in P\backslash Q$ such that $c$ is normal modulo $Q$. The algebra
is {\em normally separated} if each such pair of prime ideals is normally
separated. In our case, a result of Goodearl, see \cite[Section 5]{good},
shows that it is enough to concentrate on the $\ch$-prime ideals. Suppose that
$A$ is a $\k$-algebra with a torus $\ch$ acting rationally. If $Q$ is any
$\ch$-invariant ideal of $A$ then an element $c$ is said to be {\em
$\ch$-normal modulo $Q$} provided that there exists $h \in \ch$ such that $ca
-h(a)c\in Q$ for all $a\in A$. Goodearl observes that in this case one may
choose the element $c$ to be an $\ch$-eigenvector. 
The algebra $A$ has {\em $\ch$-normal
separation} provided that for each pair of $\ch$-prime ideals $Q\ssneq P$
there exists an element $c\in P\backslash Q$ such that $c$ is $\ch$-normal
modulo $Q$. \\

A slightly weaker notion, also introduced by
Goodearl, is that of {\em normal $\ch$-separation}. The algebra $A$ has normal
$\ch$-separation provided that for each pair of $\ch$-primes $Q\ssneq P$ there
is an $\ch$-eigenvector $c \in P\backslash Q$ which is normal modulo $Q$. 
Goodearl shows that in the situation that we are
considering, normal $\ch$-separation implies normal separation, see
\cite[Theorem 5.3]{good}. 

Notice that, as explained in paragraph 5.1 of \cite{good}, the action of $\ch$ induces a grading on 
$A$ by a suitable free abelian group. 
Using this grading, it is easy to see that $A$ has normal $\ch$-separation if and only if for each pair 
of $\ch$-primes $Q\ssneq P$ there
is an element $c \in P\backslash Q$ whose image in $A/Q$ is normal and an $\ch$-eigenvector. 
This fact will be freely used in the sequel.\\

Recall, from \cite[Definition 2.5]{llr}, the definition of a Cauchon
extension.
Let $A$ be a domain that is a noetherian $\k$-algebra and let
$R=A[X;\sigma,\delta]$ be a skew polynomial extension of $A$. We say that $R
=A[X;\sigma,\delta]$ is a {\em Cauchon Extension} provided that

\begin{itemize}

\item $\sigma$ is a $\k$-algebra automorphism of $A$ and
$\delta$ is a $\k$-linear locally nilpotent 
$\sigma$-derivation of $A$. Moreover we assume that there exists 
$q \in \k^*$ which is not a root of
unity such that $\sigma \circ \delta = q \delta \circ \sigma$.

\item There exists an abelian group $\ch$ which acts on $R$ by
$\k$-algebra 
automorphisms such that $X$ is an $\ch$-eigenvector and $A$ is
$\ch$-stable.

\item $\sigma$ coincides with the action on $A$ of an element $h_0 \in \ch$.

\item Since $X$ is an $\ch$-eigenvector and since $h_0 \in \ch$, there
 exists $\lambda_0 \in \k^*$ such that $h_0.X=\lambda_0
 X$. We assume that $\lambda_0$ is not a root of unity.

\item Every $\ch$-prime ideal of $A$ is completely prime. 

\end{itemize}


\begin{lemma} 
Suppose that $R = A[X;\sigma,\delta]$ is a Cauchon extension. Moreover, assume
that $\ch$ is a torus and that the action of $\ch$ on $R$ is rational. If $R$
has $\ch$-normal separation then $A$ has $\ch$-normal separation.
\end{lemma} 

\begin{proof}
First, note that $\{X^n\}$ is an Ore set in $R$, by \cite[Lemme 2.1]{c1};
and so we can form the Ore localization
$\widehat{R}:=RS^{-1}=S^{-1}R$. As $X$ is an $\ch$-eigenvector, the rational
action of $\ch$ on $R$ extends to a rational action on $\widehat{R}$. 
We claim that $\widehat{R}$ has $\ch$-normal separation. Suppose that $Q\ssneq
P$ are $\ch$-prime ideals of
$\widehat{R}$. Then $Q\cap R\ssneq P\cap R$ are distinct $\ch$-prime ideals of
$R$. Thus, there exist an element $c\in (P\cap R)\backslash(Q\cap R)$ and an
element $h\in\ch$ such that $cr-h(r)c \in Q\cap R$ for all $r\in R$ . 
In particular, $cX - \lambda Xc = cX-h(X)c \in Q\cap R$ for some
$\lambda\in\k^*$, as $X$ is an $\ch$-eigenvector. From this it is easy to
calculate that $(\lambda X)^{-k}c-cX^{-k}\in Q$. Now, let $y=rX^{-k}$ be an
element of $\widehat{R}$. Then, working modulo $Q$, we calculate 
\[
cy=crX^{-k} =h(r)(\lambda X)^{-k}c = h(r)h(X^{-k})c = h(rX^{-k})c = h(y)c; 
\]
so that $\widehat{R}$ has $\ch$-normal separation, as claimed.

For each $a\in A$, set 

$$\theta (a) = \sum_{n=0}^{+ \infty} \frac{(1-q)^{-n}}{[n]!_q}
\delta^n \circ \sigma^{-n} (a) X^{-n} \in\widehat{R}  $$

(Note that $\theta(a)$ is a well-defined element of $\wr$, since
$\delta$ is locally nilpotent, $q$ is not a root of unity, 
and $0\neq 1-q\in \k$.) 

The following facts are established in \cite[Section 2]{c1}. The map
$\theta:A \goesto \wr$ is a $\k$-algebra monomorphism. Let
$A[Y;\sigma]$ be a skew polynomial extension. Then $\theta$ extends to a
monomorphism $\theta: A[Y;\sigma] \goesto \wr$ with $\theta(Y) = X$.
Set $B = \theta(A)$ and $T = \theta(A[Y;\sigma]) \sse \wr$. Then $T =
B[X;\alpha]$, where $\alpha$ is the automorphism of $B$ defined by
$\alpha(\theta(a)) = \theta(\sigma(a))$.

The element $X$ is a normal element in $T$, and so the set $S$ is an Ore set
in $T$ and Cauchon shows that $TS^{-1}= S^{-1}T = \wr$. Thus, $\wr=
B[X,X^{-1};\alpha]$. Also, the $\ch$-action transfers to $B$ via $\theta$, by
\cite[Lemma 2.6]{llr}. Note, in particular, that $\alpha$ coincides with the action of an element of $\ch$ on $B$.

Thus, it is enough to show that $B$ has $\ch$-normal separation, given that 
$B[X,X^{-1};\alpha]$ has $\ch$-normal separation. 

Let $Q\ssneq P$ be $\ch$-prime ideals of $B$. Set $\widehat{Q} =
\oplus_{i\in\mz}\,QX^i$ and $\widehat{P} = \oplus_{i\in\mz}\,PX^i$. Then
$\widehat{Q}\cap B = Q$ and $\widehat{P}\cap B = P$, and it follows that
$\widehat{Q}\ssneq\widehat{P}$ are $\ch$-prime ideals in $B[X,X^{-1};\alpha]$,
see \cite[Theorem 2.3]{llr}. As $B[X,X^{-1};\alpha]$ has $\ch$-normal
separation, there is an element $c\in \widehat{P}\backslash\widehat{Q}$ and an
element $h\in\ch$ such that $cs -h(s)c \in \widehat{Q}$, for each $s\in
B[X,X^{-1};\alpha]$. Now, write $c = \sum_{i\in\mz}\,c_iX^i$. Note that each
$c_i\in P$ and at least one $c_i\not\in Q$, say $c_{i_0}\not\in Q$. Let $b\in
B$. Then, $cb-h(b)c \in \widehat{Q}$. Therefore, 
$
\sum_i\, c_iX^ib - h(b)c_iX^i \in \widehat{Q}$; and so
\[
\sum_i\, (c_i\alpha^i(b) - h(b)c_i)X^i \in \widehat{Q}
\]
As $\widehat{Q} = \oplus_{i\in\mz}\,QX^i$, this forces 
$ c_i\alpha^i(b) - h(b)c_i\in Q$ 
for each $i$, and, in particular,
$c_{i_0}\alpha^{i_0}(b) - h(b)c_{i_0}\in Q$. 
As $b$ was an arbitrary element 
of $B$, we may replace $b$ by $\alpha^{-i}(b)$ to obtain
\[
c_{i_0}b - h\alpha^{-i}(b)c_{i_0}\in Q
\]
As  
$\alpha$ coincides with the action of an element of $\ch$ on $B$, 
this produces an element $h_{i_0}\in\ch$ such that 
\[
c_{i_0}b - h_{i_0}(b)c_{i_0}\in Q, 
\]
as required to show that $B$ has $\ch$-normal separation. 
\qed\end{proof}


\begin{theorem}
Let $\lambda = (\lambda_1, \lambda_2, \dots, \lambda_m)$ be a partition with
$n\geq \lambda_1 \geq \lambda_2\geq \dots \geq \lambda_m \geq 0$ and let
$\alambda$ be the corresponding partition subalgebra of generic quantum
matrices. Then $\alambda$ has $\ch$-normal separation.
\end{theorem}

\begin{proof} 
Let $\mu=(n,\dots,n)$ ($m$ times); so that $Y_\mu$ is an $m 
\times n$ rectangle. Then $\ca_\mu= \coq(M_{m,n}(\k))$; 
and so $\ca_\mu$ has
$\ch$-normal separation, by \cite[Th\'eor\`eme 6.3.1]{c2}. We can construct
$\ca_\mu$ from $\ca_\lambda$ by adding the missing variables
$x_{ij}$ in lexicographic order. At each stage, the extension is a Cauchon
extension. Thus, $\ca_\lambda$ has $\ch$-normal separation, by repeated
application of the previous lemma. \qed\end{proof} 


\begin{corollary}
Let $\lambda = (\lambda_1, \lambda_2, \dots, \lambda_m)$ be a partition with
$n\geq \lambda_1 \geq \lambda_2\geq \dots \geq \lambda_m \geq 0$ and let
$\alambda$ be the corresponding partition subalgebra of generic quantum
matrices. Then $\alambda$ has normal $\ch$-separation and normal separation.
\end{corollary} 

\begin{proof} We have seen earlier that $\ch$-normal separation implies normal
$\ch$-separation. Normal separation now follows from \cite[Theorem 5.3]{good}.
\qed\end{proof}

\begin{corollary}\label{partition-catenary}
Let $\lambda = (\lambda_1, \lambda_2, \dots, \lambda_m)$ be a partition with
$n\geq \lambda_1 \geq \lambda_2\geq \dots \geq \lambda_m \geq 0$ and let
$\alambda$ be the corresponding partition subalgebra of generic quantum
matrices. Then $\alambda$ is catenary.
\end{corollary}

\begin{proof} This follows from the previous results and \cite[Theorem 0.1]{yz} which
states that if $A$ is a normally separated filtered $\k$-algebra such that
${\rm gr}(A)$ is a noetherian connected graded $\k$-algebra with enough normal
elements then $\spec(A)$ is catenary. (For the notion of an algebra with
enough normal elements see \cite{z}.)\qed\end{proof}~\\

Note that it is also possible to deduce this result from \cite[Theorem
1.6]{glcat}


\section{Quantum Schubert cells} \label{section-cells} 


Quantum Schubert cells in the quantum grassmannian are obtained from quantum
Schubert varieties via the process of noncommutative dehomogenisation
introduced in \cite{klr}. Recall that if $R = \oplus R_i$ is an $\mn$-graded
$\k$-algebra and $x$ is a regular homogeneous normal element of $R$ of degree
one, then the {\em dehomogenisation} of $R$ at $x$, written $\dhom(R,x)$, is
defined to be the zero degree subalgebra $S_0$ of the $\mz$-graded algebra
$S:= R[x^{-1}]$. If $R$ is generated as a $\k$-algebra by $a_1, a_2, \dots,
a_s$ and each $a_i$ has degree one, then it is easy to check that $\dhom(R,x)
= k[a_1x^{-1}, \dots, a_sx^{-1}]$. 
\\

If $\sigma$ denotes the automorphism of $S$ given by $\sigma(s) =
xsx^{-1}$ for $s\in S$ then $\sigma$ induces an automorphism of
$S_0$, also denoted by $\sigma$, and there is an isomorphism
$$\theta:\dhom(R,x)[y,y^{-1}; \sigma] \goesto R[x^{-1}]$$
which is the identity on
$\dhom(R,x)$ and sends $y$ to $x$.
\\


Let $\gamma\in\Pi_{m,n}$. 
Recall from Remark~\ref{remark-single-min-elt} that 
$S(\gamma) = {\mathcal O}_q(G_{m,n}(\k))/\langle \Pi_{m,n}^\gamma \rangle$ 
and that $\wbar{\gamma}$
is a homogeneous regular normal element of degree one in $S(\gamma)$. It
follows that we can form the localisation $S(\gamma)[\wbar{\gamma}^{-1}]$ and
that $S(\gamma)\sse S(\gamma)[\wbar{\gamma}^{-1}]$.


\begin{definition}The {\em quantum Schubert cell} associated to the quantum
minor $\gamma$ is denoted by $\sogamma$ and is defined to be
$\dhom(\sgamma,\wbar{\gamma})$.
\end{definition} 


\begin{remark} 
{\rm In the classical case when $q =1$, it can be seen that 
this definition coincides with the usual definition of Schubert cells, 
as discussed, for example, in \cite[Section 9.4]{fult}}
\end{remark}

It follows from the definition that $\sogamma$ is generated by the elements
$\wbar{x}\,\wbar{\gamma}^{-1}$, for 
$x\in \Pi_{m,n}\setminus(\Pi_{m,n}^\gamma\cup\{\gamma\})$. 
However, these elements are not
independent; so we will pick out a better generating set for the quantum
Schubert cell.


This is achieved by using the quantum ladder matrix
algebras introduced in \cite[Section 3.1]{lr3}. Let us recall the definition.
To each $\gamma=(\gamma_1,\dots,\gamma_m)\in\Pi_{m,n}$, with $1 \le \gamma_1<
\dots < \gamma_m\le n$, we associate the substet ${\mathcal L}_\gamma$ of
$\{1,\dots,m\} \times \{1,\dots,n\}$ defined by \[ {\mathcal L}_\gamma =
\{(i,j) \in \{1,\dots,m\} \times \{1,\dots,n\} \;\tq\; j > \gamma_{m+1-i}
\quad\mbox{and}\quad j\neq \gamma_\ell \quad\mbox{for}\quad 1 \le \ell \le
m\}, \] which we call the {\em ladder} associated with $\gamma$.


Consider the quantum minors $m_{ij}$ defined by
$m_{ij}:=[\{\gamma_1,\dots,\gamma_m\} \setminus\{\gamma_{m+1-i}\} \cup
\{j\}]$, for each $(i,j)\in{\mathcal L}_\gamma$. 
These are the quantum minors that are
above $\gamma$ in the standard order and differ from $\gamma$ in exactly one
position. Denote the set of these quantum minors by $\mcgamma$. 


\begin{proposition}
$\sogamma = \k[\wbar{m_{ij}}\,\wbar{\gamma}^{-1} \mid m_{ij}\in\mcgamma]$
\end{proposition} 


\begin{proof} 
In the proof of \cite[Theorem 3.1.6]{lr3} it is shown that
$S(\gamma)[\wbar{\gamma}^{-1}]$ is generated by the elements $\wbar{\gamma},
\wbar{\gamma}^{-1}$ and the $\wbar{m_{ij}}$. The Schubert cell $\sogamma$ is
the degree zero part of this algebra. As $\wbar{\gamma}$ and $\wbar{m_{ij}}$
commute up to scalars, see \cite[Lemma 3.1.4(v)]{lr3}, 
it is easy to check that $\sogamma$ is generated by
$\wbar{m_{ij}}\,\wbar{\gamma}^{-1}$, as required.
\qed\end{proof}\\


Set $\widetilde{m_{ij}}:= \wbar{m_{ij}}\,\wbar{\gamma}^{-1}$. 


\begin{lemma}\label{lemma-action-on-mijtilde}
There is an induced action of $\ch = (\k^*)^n$ on $\sogamma$ given by 
\[
(\alpha_1, \alpha_2, \dots, \alpha_n)\circ\widetilde{m_{ij}} := 
\alpha^{-1}_{\gamma_{m+1-i}}\alpha_j\widetilde{m_{ij}}.
\]
\end{lemma} 

\begin{proof} This follows immediately from the fact that 
\[
\widetilde{m_{ij}} = [\overline{\{\gamma_1, \dots,
\gamma_m\}\backslash\{\gamma_{m+1-i}\}\cup\{j\}]}\,\overline{[\gamma_1,\dots,
\gamma_m]}^{-1}.
\]
\qed\end{proof}~\\ 


We now need to establish the commutation relations between the
$\widetilde{m_{ij}}$. 


\begin{definition} -- \label{def-q-ladder-matrix} 
Let $\gamma=(\gamma_1,\dots,\gamma_m)\in\Pi_{m,n}$, with $1 \le \gamma_1<
\dots < \gamma_m\le n$. The quantum ladder matrix algebra associated with
$\gamma$, denoted ${\mathcal O}_q(M_{m,n,\gamma}(\k))$, is the $\k$-subalgebra
of ${\mathcal O}_q(M_{m,n}(\k))$ generated by the elements $x_{ij}\in
{\mathcal O}_q(M_{m,n}(\k))$ such that $(i,j)\in{\mathcal L}_\gamma$.
\end{definition}


The following example, taken from \cite{lr3} will help clarify this
definition.


\begin{example} -- \rm
We put $(m,n)=(3,7)$ and
$\gamma=(\gamma_1,\gamma_2,\gamma_3)=(1,3,6)\in\Pi_{3,7}$. In the $3 \times 7$
generic matrix ${\bf X}=\left(x_{ij}\right)$ associated with 
${\mathcal O}_q(M_{3,7}(\k))$, put a bullet on each row as follows: 
on the first row,
the bullet is in column $6$ because $\gamma_3$ is $6$, on the second row, the
bullet is in column $3$ because $\gamma_2$ is $3$ and on the third row, 
the bullet is
in column $1$ because $\gamma_1 = 1$. 
Now, in each position which is to the left of a bullet, or
which is below a bullet, put a star. To finish, place 
$x_{ij}$ in any  position $(i,j)$ that has not yet been filled. 
We obtain 
\[
\left(
\begin{array}{ccccccc}
 \ast & \ast & \ast & \ast & \ast & \bullet & x_{17} \cr   
 & & & & & & \cr 
\ast & \ast & \bullet & x_{24} & x_{25} & \ast & x_{27} \cr
 & & & & & & \cr
\bullet & x_{32} & \ast & x_{34} & x_{35} & \ast & x_{37} \cr
\end{array}
\right) .
\]
By definition, the quantum ladder matrix algebra associated 
with $\gamma=(1,3,6)$ is the subalgebra of ${\mathcal O}_q(M_{3,7}(\k))$ 
generated by the elements
$x_{17}, x_{24}, x_{25}, x_{27}, x_{32}, x_{34}, x_{35}, x_{37}$.  
\end{example}


Notice that if we rotate the matrix above through $180^\circ$ then the
$x_{ij}$ involved in the definition of ${\mathcal O}_q(M_{3,7,\gamma}(\k))$
sit naturally in the Young Diagram of the partition $\lambda = (4,3,1)$. We
will return to this point later.


\begin{lemma} 
The quantum Schubert cell $\sogamma$ is isomorphic to the quantum ladder
matrix algebra ${\mathcal O}_q(M_{m,n,\gamma}(\k))$. 
\end{lemma} 


\begin{proof}
Lemma 3.1.4 of \cite{lr3} shows that the commutation relations for the
$m_{ij}$ are the same as the commutation 
relations for corresponding variables
$x_{ij}$ in the quantum ladder matrix algebra
${\mathcal O}_q(M_{m,n,\gamma}(\k))$.
As $\gamma m_{ij} = qm_{ij}\gamma$, for each $i, j$, by 
\cite[Lemma 3.1.4(v)]{lr3}, it follows that the 
$\widetilde{m_{ij}}$ satisfy the same relations. 
Thus there is an epimorphism from ${\mathcal O}_q(M_{m,n,\gamma}(\k))$
onto $\sogamma$.
A comparison of Gelfand-Kirillov dimensions similar to that used in
\cite[Theorem 3.1.6]{lr3} now shows that this epimorphism is in fact an
isomorphism.
\qed\end{proof}


\begin{theorem}\label{theorem-cell-partition} 
The quantum Schubert cell $\sogamma$ is (isomorphic to) a partition subalgebra
of ${\mathcal O}_{q^{-1}}(M_{m,n-m}(\k))$.
\end{theorem}


\begin{proof} 
For any $n$, there is an isomorphism $\delta:{\mathcal O}_{q}(M_n(\k))
\goesto{\mathcal O}_{q^{-1}}(M_n(\k))$ defined by $\delta(x_{ij}) =
x_{n+1-i,n+1-j}$, see the proof of \cite[Corollary 5.9]{gl}. The isomorphism
$\delta$ can be used to convert quantum ladder matrix algebras into
partition subalgebras. As Schubert cells are isomorphic to quantum 
ladder matrix algebras, the result follows. 
\qed\end{proof}\\


The isomorphism described in the previous result carries over the $\ch$-action
on $\sogamma$ to the partition subalgebra, and this induced action acts via
row and column multiplications. After suitable re-numbering of the components
of $\ch$, this action coincides with the action discussed at the beginning of
Section~\ref{section-partitions}. As a consequence of
Theorem~\ref{theorem-cell-partition}, the results obtained in
Section~\ref{section-partitions} apply to quantum Schubert cells. In
particular, the following results hold. 


\begin{theorem}\label{theorem-hspecsogamma-bijection}
Let $\lambda = (\lambda_1, \lambda_2, \dots, \lambda_m)$ be the partition with
$n\geq \lambda_1 \geq \lambda_2\geq \dots \geq \lambda_m \geq 0$ defined by
$\lambda_i + \gamma_i = n-m+i$ and let $Y_\lambda$ be the corresponding Young
diagram. Then the $\ch$-prime spectrum of $\sogamma$ is in bijection with the
set of Cauchon diagrams on the Young diagram, $Y_\lambda$, as described in
Theorem~\ref{theopartitionalgebra}.
\end{theorem} 


\begin{theorem}\label{theorem-schubert-cell-normal}
The quantum Schubert cell $\sogamma$ has $\ch$-normal separation, normal
$\ch$-separation and normal separation.
\end{theorem}

\begin{corollary}
The quantum Schubert cell $\sogamma$ is catenary.
\end{corollary}


\section{The prime spectrum of the quantum grassmannian} 


In this section, we use the quantum Schubert cells to obtain information
concerning the prime spectrum of the quantum grassmannian. We show that, in
the generic case, where $q$ is not a root of unity, all primes are completely
prime and that there are only finitely many primes that are invariant under
the natural torus action on the quantum grassmannian. By using a result of
Lauren Williams, we are able to count the number of $\ch$-primes. Also, we are
able to show that the quantum grassmannian is catenary.\\


Note that the following result is valid for any $q\neq 0$. 


\begin{theorem}
Let $P$ be a prime ideal of $\oqgmn$ with $P\neq \left<\Pi\right>$; that is,
$P$ is not the irrelevant ideal. Then there is a unique $\gamma$ in $\Pi$ with
the property that $\gamma\not\in P$ but $\pi\in P$ for all
$\pi\not\geq_{\st}\gamma$.
\end{theorem}


\begin{proof} If $\Pi\sse P$ then $P$ is the irrelevant ideal. Otherwise, there
exists $\gamma\in\Pi$ with $\gamma\not\in P$. Choose such a $\gamma$ that is
minimal in $\Pi$ with this property. Then $\lambda \in P$ for all $\lambda
<_{\st}\gamma$. 

Note that $\left< \{\lambda \mid \lambda <_{\st}\gamma\}\right>
 \subseteq P$ and that $\gamma$ is normal modulo 
 $\left< \{\lambda \mid \lambda <_{\st}\gamma\}\right>$, by 
\cite[Lemma 1.2.1]{lr2}; so that
$\gamma$ is normal modulo $P$.

Suppose that $\pi\not\geq_{\st}\gamma$. If $\pi <_{\st}\gamma$ then $\pi\in
 P$ by the choice of $\gamma$. 
 If not, then $\pi$ and $\gamma$ are not comparable.
Thus, we can write
\[
\pi\gamma=\sum  k_{\lambda\mu}\lambda\mu 
\]
with
$  k_{\lambda\mu}\in\k$ while
$ \lambda,\mu\in \Pi$ with  $\lambda <_\st \gamma$, by 
\cite[Theorem 3.3.8]{lr2}.

It follows that $\pi\gamma \in P$. Thus, $\pi\in P$, since $\gamma\not\in P$
and $\gamma$ is normal modulo $P$.

This shows that there is a $\gamma$ with the required properties. It is easy
to check that there can only be one such $\gamma$.
\qed\end{proof} \\


This enables us to give a decomposition of the prime spectrum,
$\spec(\oqgmn)$. Set $\spec_\gamma(\oqgmn)$ to be the set of prime ideals $P$
such that $\gamma\not\in P$ while $\pi\in P$ for all 
$\pi \not\geq_{\st}\gamma$. The
previous result shows that 
\[ 
\spec(\oqgmn) =
\bigsqcup_{\gamma\in\Pi}\,\spec_\gamma(\oqgmn)\;\bigsqcup\;\left<\Pi\right>. 
\]


We now re-instate our convention that $q$ is not a root of unity. 


\begin{theorem} Let $q$ be a non root of unity. 
Then all prime ideals of the quantum grassmannian $\oqgmn$ are
completely prime.
\end{theorem}


\begin{proof}
Let $P$ be a prime ideal of $\oqgmn$. If $P = \left<\Pi\right>$ then $\oqgmn/P
\cong \k$; so $P$ is completely prime.

Otherwise, suppose that $P\in\spec_\gamma(\oqgmn)$. In this case, $\wbar{P} =
P/ \langle \Pi_{m,n}^\gamma \rangle$ is a prime ideal in $S(\gamma) =
{\mathcal O}_q(G_{m,n}(\k))/\langle \Pi_{m,n}^\gamma \rangle$ and it is enough
to show that $\wbar{P}$ is completely prime. Set $T:=
\sgamma[\wbar{\gamma}^{-1}]$. Then $\wbar{P}T$ is a prime ideal of $T$ and
$\wbar{P}T \cap \sgamma = \wbar{P}$. Thus $\sgamma/\wbar{P}\sse T/ \wbar{P}T$ and so
it is enough to show that $\wbar{P}T$ is completely prime.

Now, the dehomogenisation isomorphism shows that $T\cong
\sogamma[y,y^{-1};\sigma]$, where $\sigma$ is the automorphism determined by
the conjugation action of $\wbar{\gamma}$, see the beginning of
Section~\ref{section-cells}.

We know that $\sogamma$ is a torsionfree CGL-extension by
Proposition~\ref{prop-partition-cgl} and Theorem~\ref{theorem-cell-partition}.
It is then easy to check that $\sogamma[y;\sigma]$ is also a torsionfree
CGL-extension. Thus, all prime ideals of $\sogamma[y;\sigma]$ are completely
prime, by \cite[Theorem II.6.9]{bg}, and it follows that all prime ideals of
$T\cong\sogamma[y,y^{-1};\sigma]$ are completely prime, as required.
\qed\end{proof}\\


Of course, the decomposition of $\spec(\oqgmn)$ above induces a similar
decomposition of $\hspec(\oqgmn)$:

\[
\hspec(\oqgmn) = 
\bigsqcup_{\gamma\in\Pi}\,\hspec_\gamma(\oqgmn)
\;\bigsqcup\;\left<\Pi\right>, 
\]
where $\hspec_\gamma(\oqgmn)$ is the set of $\ch$-prime ideals $P$ such that
$\gamma\not\in P$ while $\pi\in P$ for all $\pi \not\geq_{\st}\gamma$.

Our next task is to show that $\hspec_\gamma(\oqgmn)$ is in natural bijection
with $\hspec(\sogamma)$ and hence in bijection with Cauchon diagrams on the
associated Young diagram $Y_\lambda$. As a consequence, we are able to
calculate the size of $\hspec(\oqgmn)$.


\begin{remark}\label{remark-hypothesis2.1-ok}
{\rm 
Recall from the beginning of Section~\ref{section-cells}  
that, for any $\gamma\in\Pi_{m,n}$, 
there is the dehomogenisation isomorphism
\[
\theta :  \sogamma[y,y^{-1};\sigma] \longrightarrow 
S(\gamma)[\wbar{\gamma}^{-1}],
\]
where $\sigma$ is conjugation by $\wbar{\gamma}$.
Hence, the action of $\ch$ on $S(\gamma)[\wbar{\gamma}^{-1}]$ transfers, 
via $\theta$, to an action on
$\sogamma[y,y^{-1};\sigma]$. 
By Lemma~\ref{lemma-action-on-mijtilde}, $\sogamma$ is stable 
under this action and it is clear that $y$ is an $\ch$-eigenvector. 
Further, let $h_0=(\alpha_1,\dots,\alpha_n) \in \ch$ be such that 
$\alpha_i=q^2$ if  $i \notin \{\gamma_1,\dots,\gamma_m\}$ and $\alpha_i=q$ 
otherwise. Then, by using 
\cite[Lemma 3.1.4(v)]{lr3} and Lemma~\ref{lemma-action-on-mijtilde}, 
it is easily 
verified that the action of $h_0$ on $\sogamma$ coincides with $\sigma$. 
In addition, $h_0(y)=q^my$, since $h_0(\wbar{\gamma})=q^m\wbar{\gamma}$. 
It follows that
$\sogamma[y,y^{-1};\sigma]$ satisfies Hypothesis 2.1 in \cite{llr}.}
\end{remark}


\begin{theorem}
Let $P\in\hspec_\gamma(\oqgmn)$; so that $P$ is an $\ch$-prime ideal of
$\oqgmn$ such that $\gamma\not\in P$, while $\pi\in P$ for all $\pi
\not\geq_{\st}\gamma$. Set $T=
\sgamma[\wbar{\gamma}^{-1}]\cong\sogamma[y,y^{-1};\sigma]$. Then the
assignment $P\mapsto\wbar{P}T\cap\sogamma$ defines an inclusion-preserving
bijection from $\hspec_\gamma(\oqgmn)$ to $\hspec(\sogamma)$, with inverse
obtained by sending $Q$ to the inverse image in $\oqgmn$ of $QT\cap\sgamma$.
(Note, we are treating the isomorphism above as an id
entification in these
assignments.)
\end{theorem}


\begin{proof} 
This follows from the conjunction of two bijections. First, standard
localisation theory shows that $ \wbar{P} = \wbar{P}T\cap S(\gamma)$; and this
gives a bijection between $\hspec_\gamma(\oqgmn)$ and $\hspec(T)$. For the
second bijection, note that $T\cong\sogamma[y,y^{-1};\sigma]$ and that the
automorphism $\sigma$ is given by the action of an element of $\ch$, see
Remark~\ref{remark-hypothesis2.1-ok}. Thus, it follows from \cite[Theorem
2.3]{llr} that there is a bijection between $\hspec(T)$ and $\hspec(\sogamma)$
given by intersecting an $\ch$-prime of $T$ with $\sogamma$. The composition
of these two bijections produces the required bijection. 
\qed\end{proof}\\


\begin{corollary}
$\hspec_\gamma(\oqgmn)$ is in bijection with the Cauchon diagrams on
$Y_\lambda$, where $\lambda$ is the partition associated with $\gamma$.
\end{corollary}

\begin{proof} 
This follows from the previous theorem and 
Theorem~\ref{theorem-hspecsogamma-bijection}.
\qed\end{proof}\\


It follows from this corollary and the partition of the $\ch$-spectrum of the
quantum grassmannian that the $\ch$-spectrum of the quantum grassmannian is
finite. This finiteness is a crucial condition needed to investigate normal
separation, Dixmier-Moeglin equivalence, etc. in the quantum case because of
the stratification theory, see, for example, \cite[Theorem 5.3]{good},
\cite[Theorem II.8.4 ]{bg}. However, in this situation, we can say much more:
we can say exactly how many $\ch$-primes there are in the quantum grassmannian
$\oqgmn$. This is one more (the irrelevant ideal $\langle\Pi\rangle$) than the
total number of Cauchon diagrams on the Young diagrams $Y_\lambda$
corresponding to the partitions $\lambda$ that fit into the partition
$(n-m)^m$. This combinatorial problem has been solved by Lauren Williams, in
\cite{w}. The following result is obtained by setting $q=1$ in the formula for
$A_{k,n}(q)$ in \cite[Theorem~4.1]{w}. 


\begin{theorem}
\[
\left|\hspec(\oqgmn)\right| = 1 + 
 \sum_{i=0}^{m-1} {n\choose i} \left( (i-m)^i(m-i+1)^{n-i} 
- (i-m+1)^i (m-i)^{n-i}\right)
\]
\end{theorem} 


\begin{proof} 
By using the results above, we see that, except for the irrelevant ideal,
each $\ch$-prime corresponds to a unique Cauchon diagram drawn on the Young
diagram $Y_\lambda$ that corresponds to the partition $\lambda$ associated to
the quantum minor $\gamma$ which determines the cell that $P$ is in. 

In \cite[Theorem 4.1]{w}, Lauren Williams has counted the number of Cauchon
diagrams on the Young diagrams $Y_{\lambda}$ that fit into the partition
$(n-m)^m$; and this count, plus one, gives the number of $\ch$-prime ideals of
$\oqgmn$.
\qed \end{proof} \\


For example, $|\hspec(\coq(G_{2,4}))| = 34$ and $|\hspec(\coq(G_{3,6}))| =
884$. (These numbers can be seen from the table in \cite[Figure 23.1]{post}.)
\\


We turn now to the questions of normal separation and catenarity. In order to
establish these properties for the quantum grassmannian, we need to use the
dehomogenisation isomorphism. Recall that the methods of \cite{llr} are
available because of Remark~\ref{remark-hypothesis2.1-ok}.


\begin{lemma}
Let $Q\ssneq P$  be $\ch$-prime ideals in $S(\gamma)$ that do not contain 
$\wbar{\gamma}$. Then, there is an $\ch$-eigenvector in $P\backslash Q$ 
that is normal modulo $Q$.
\end{lemma}


\begin{proof}
Let $Q\subsetneqq P$ be $\ch$-prime ideals in $S(\gamma)$
that do not contain $\wbar{\gamma}$. Set $T:= S(\gamma)
[\wbar{\gamma}^{-1}]$ and observe that there is an induced action of the torus
$\ch$ on $T$, because $\gamma$ is an $\ch$-eigenvector. Note that $P = PT\cap
S(\gamma)$ and $Q = QT\cap S(\gamma)$; so $QT\subsetneqq PT$ are $\ch$-prime
ideals in $T$. Now, set $P_0:=PT\cap\sogamma$ and $Q_0:=QT\cap\sogamma$ (here,
we are treating the isomorphism $T\cong \sogamma[y,y^{-1};\sigma]$ as an
identification) and note that 
$PT = \oplus_{n\in\mz} P_0\,y^n$ 
and 
$QT = \oplus_{n\in\mz} Q_0\,y^n$; 
so $Q_0\subsetneqq P_0$ are $\ch$-prime ideals of
$\sogamma$, see Remark~\ref{remark-hypothesis2.1-ok} 
and \cite[Theorem 2.3]{llr}. 
These observations make it clear that 

\[ \frac{\sogamma}{Q_0}[y,y^{-1};\sigma]\quad\cong
\quad\frac{T}{QT}\quad\cong\quad\frac{S(\gamma)}{Q}[\wbar{\gamma}^{-1}].
\]
As usual, $\wbar{\sogamma}$ will denote $\sogamma/Q_0$, etc.

The quantum Schubert cell $\sogamma$ has $\ch$-normal separation, by
Theorem~\ref{theorem-schubert-cell-normal}. Thus, there exists an
$\ch$-eigenvector $c\in P_0\backslash Q_0$ and an element $h\in\ch$ such that
$ca -h(a)c\in Q_0$ for all $a\in\sogamma$. Recall that the action of $\sigma$
coincides with the action of an element $h_y$ of $\ch$; so that $yc =h_y(c)y=
\lambda cy$ for some $\lambda\in\k^*$. It follows that $\wbar{c}$ is normal in
$T/QT$. Define $\sigmac:T/QT\goesto T/QT$ by $\wbar{c}\wbar{t}\ =
\sigmac(\wbar{t})\wbar{c}$ for all $t\in T$. Note that
$\sigmac|_{\wbar{\sogamma}} = h|_{\wbar{\sogamma}}$ and that $\sigmac(y) =
\lambda^{-1} y$. 

We claim that $\sigma_c(S(\gamma)/Q) = S(\gamma)/Q$;
so that $\sigmac$ induces  an
isomorphism on this algebra. In order to see this, note that $S(\gamma)/Q$ is
generated as an algebra by the images of the quantum minors $[\alpha_1,
\dots,\alpha_m]$ for $[\alpha_1, \dots,\alpha_m] \geq \gamma$.
Now, $\wbar{[\alpha_1,
\dots,\alpha_m]}\,\wbar{\gamma}^{-1}\in \wbar{\sogamma}$, 
because $[\alpha_1, \dots,\alpha_m]\gamma^{-1}$ has degree zero in $T$ 
so that $[\alpha_1, \dots,\alpha_m]\gamma^{-1}\in\sogamma$.
Thus, recalling that
$\wbar{\gamma}$ is identified with $y$ under the isomorphisms above, 
\begin{align*}
\sigmac(\wbar{[\alpha_1, \dots,\alpha_m]}) 
    &=
\sigmac(\wbar{[\alpha_1,\dots,\alpha_m]}\,\wbar{\gamma}^{-1}) 
\sigmac(\wbar{\gamma})
    =
h(\wbar{[\alpha_1, \dots,\alpha_m]}\,\wbar{\gamma}^{-1})(\lambda^{-1} y)\\
    &=
\mu\, \wbar{[\alpha_1, \dots,\alpha_m]}\,\wbar{\gamma}^{-1}(\lambda^{-1} y)
    =
(\mu\lambda^{-1})\wbar{[\alpha_1, \dots,\alpha_m]}\,\wbar{\gamma}^{-1}y\\ 
    &=
(\mu\lambda^{-1} )\wbar{[\alpha_1, \dots,\alpha_m]},    
\end{align*}
where the existence of $\mu\in\k^*$ is guaranteed because $h$ is acting as a
scalar on the element 
$\wbar{[\alpha_1, \dots,\alpha_m]}\,\wbar{\gamma}^{-1}\in\sogamma/Q_0$. 
The claim follows. 

There exists $d\geq 0$ such that $\wbar{c}\,\wbar{\gamma}^{d}\in
S(\gamma)/Q$. It is obvious that $c\wbar{\gamma}^{d}$ is an
$\ch$-eigenvector, because each of $c$ and $\gamma$ is an $\ch$-eigenvector.
Also, $c\wbar{\gamma}^{d} \in P\backslash Q$. Finally,
$\wbar{c}\,\wbar{\gamma}^{d}$ is normal in $S(\gamma)/Q$, because
$S(\gamma)/Q$ is invariant under conjugation by each of $\wbar{c}$ and
$\wbar{\gamma}$. 
\qed
\end{proof}


\begin{theorem}
The quantum grassmannian $\oqgmn$ has normal $\ch$-separation and hence normal
separation.
\end{theorem} 

\begin{proof} Suppose that $Q\ssneq P$ are $\ch$-prime ideals of $\oqgmn$. Suppose
that $Q\in\spec_\gamma(\oqgmn)$. If $\gamma\in P$, then $P$ contains the
$\ch$-eigenvector $\gamma$. 

Otherwise, $\gamma\not\in P$ and $P\in\spec_\gamma(\oqgmn)$. 
In this case, it is
enough to show that there is a $\ch$-eigenvector in
$\wbar{P}\backslash\wbar{Q}$ which is normal modulo $\wbar{Q}$, 
where $\wbar{P}= P/ \langle \Pi_{m,n}^\gamma
\rangle$ and $\wbar{Q} =Q/ \langle \Pi_{m,n}^\gamma \rangle$ 
are $\ch$-prime ideals in $S(\gamma)$. 
However, this has been done in the previous lemma.
\qed\end{proof}


\begin{theorem} The quantum grassmannian 
$\oqgmn$ is catenary.
\end{theorem}

\begin{proof} As in Corollary~\ref{partition-catenary}, 
this follows from the previous results and \cite[Theorem 0.1]{yz}.
\qed\end{proof}


\begin{remark}
{\rm It is obvious from the style of proof of the preceding results that there
is now a good strategy for producing results concerning the quantum
grassmannian: first, establish the corresponding results for partition
subalgebras of quantum matrices, and then use the theory of quantum Schubert
cells and noncommutative dehomogenisation to obtain the result in the quantum
grassmannian. We leave any further developments for interested readers. }

\end{remark} 



\section{Concluding remark.}


We end  this work by stressing some important connections between the
results established in Section 4 above, and recent work of Postnikov in
total positivity, see \cite{post}. \\

Let $M_{m,n}^+(\mr)$ denote the space of $m \times n$ real matrices of rank
$m$ and whose $m \times m$ minors are nonnegative. The group $GL_m^+(\mr)$ of
$m \times m$ real matrices of positive determinant act naturally on
$M_{m,n}^+(\mr)$ by left multiplication. The corresponding quotient space
$G_{m,n}^+(\mr)=M_{m,n}^+(\mr)/GL_m^+(\mr)$ is the {\em totally nonnegative}
grassmannian of $m$ dimensional subspaces in $\mr^n$. One can define a
cellular decomposition of $G_{m,n}^+(\mr)$ by specifying, for each element of
$G_{m,n}^+(\mr)$, which $m \times m$ minors are zero and which are strictly
positive. The corresponding cells are called the {\em totally nonnegative
cells} of $G_{m,n}^+(\mr)$. In \cite{post}, Postnikov shows that totally
nonnegative cells in $G_{m,n}^+(\mr)$ are in bijection with the Cauchon
diagrams on partitions fitting into the partition $(n-m)^m$. For further
details, see Sections 3 and 6 in \cite{post}.\\

Hence, by the results in Section 4 above, the set of totally nonnegative cells
of $G_{m,n}^+(\mr)$ is in one-to-one correspondance with the set
of $\ch$-prime ideals of $\oqgmn$ distinct from the augmentation ideal. We
believe it would be interesting to understand this coincidence and we intend
to pursue this theme in a subsequent paper.


\vskip 1cm
\newpage 

\noindent 
S Launois:\\ 
Institute of Mathematics, Statistics and Actuarial Science,\\
University of Kent at Canterbury,\\ CT2 7NF,\\ UK\\~\\
E-mail: S.Launois@kent.ac.uk \\

~\\
T H Lenagan: \\
Maxwell Institute for Mathematical Sciences\\
School of Mathematics, University of Edinburgh,\\
James Clerk Maxwell Building, King's Buildings, Mayfield Road,\\
Edinburgh EH9 3JZ, Scotland, UK\\~\\
E-mail: tom@maths.ed.ac.uk \\

~\\
L Rigal: \\
Universit\'e Jean-Monnet (Saint-\'Etienne), \\
Facult\'e des Sciences et Techniques, \\
D\'e\-par\-te\-ment de Math\'ematiques,\\
23 rue du Docteur Paul Michelon,\\
42023 Saint-\'Etienne C\'edex 2,\\France\\~\\
E-mail: Laurent.Rigal@univ-st-etienne.fr\\


\end{document}